\documentclass[11pt]{article}

\usepackage[T1]{fontenc}
\usepackage[english]{babel}
\usepackage[a4paper,margin=1in]{geometry}
\usepackage{amsmath,amssymb,amsthm}
\usepackage{booktabs}
\usepackage{algorithm}
\usepackage{algpseudocode}
\usepackage{tikz}
\usepackage{xcolor}
\usepackage[normalem]{ulem}
\usepackage[colorlinks=true,linkcolor=blue,citecolor=blue,urlcolor=blue]{hyperref}

\theoremstyle{plain}
\newtheorem{theorem}{Theorem}
\newtheorem{proposition}{Proposition}
\newtheorem{lemma}{Lemma}

\theoremstyle{definition}
\newtheorem{definition}{Definition}
\newtheorem{example}{Example}
\theoremstyle{remark}

\newcommand{\exP}{\operatorname{ex}_{\mathcal P}}

\newcommand{\Ball}{\mathcal B}
\newcommand{\Near}{\mathcal N}
\newcommand{\defterm}[1]{\textit{#1}}

\newif\ifshowchanges
\showchangesfalse

\ifshowchanges
  
  \newcommand{\delete}[1]{\textcolor{blue}{\sout{#1}}}
  
\else
  
  \newcommand{\delete}[1]{}
  
\fi

\newcommand{\CertOne}{\texorpdfstring{\hyperref[cert:C1]{\textup{(C1)}}}{(C1)}}
\newcommand{\CertTwo}{\texorpdfstring{\hyperref[cert:C2]{\textup{(C2)}}}{(C2)}}
\newcommand{\CertThree}{\texorpdfstring{\hyperref[cert:C3]{\textup{(C3)}}}{(C3)}}
\newcommand{\CertFour}{\texorpdfstring{\hyperref[cert:C4]{\textup{(C4)}}}{(C4)}}
\newcommand{\CertFive}{\texorpdfstring{\hyperref[cert:C5]{\textup{(C5)}}}{(C5)}}
\newcommand{\CertSix}{\texorpdfstring{\hyperref[cert:C6]{\textup{(C6)}}}{(C6)}}

\hypersetup{
  pdftitle={An improved upper bound for the planar Turan number of C8},
  pdfsubject={Planar Turan numbers; C8-free planar graphs},
  pdfauthor={Xuqing Bai, Weichan Liu, Xiangxiang Nie, Xin Zhang},
  pdfkeywords={planar Turan number, C8-free planar graph, discharging, computer-assisted proof},
  pdflang={en-US}
}

\title{An improved upper bound for the planar Tur\'an number of \texorpdfstring{$C_8$}{C8}}
\author{%
Xuqing Bai\textsuperscript{1}\thanks{Supported by the National Natural Science
Foundation of China (Grant No.\,12501495).}\quad
Weichan Liu\textsuperscript{2}\thanks{Supported by the Postdoctoral Fellowship
Program of CPSF (Grant No.\,GZC20252020) and the China Postdoctoral Science
Foundation (Grant No.\,2025M783118).}\quad
Xiangxiang Nie\textsuperscript{3}\quad
Xin Zhang\textsuperscript{1}\thanks{Corresponding author.}\\[0.5em]
\textsuperscript{1}School of Mathematics and Statistics, Xidian University, Xi'an, 710071, China\\
\textsuperscript{2}School of Mathematics, Shandong University, Jinan, 250100, China\\
\textsuperscript{3}Data Science Institute, Shandong University, Jinan, 250100, China\\[0.5em]
\texttt{baixuqing@xidian.edu.cn},
\texttt{wcliu@sdu.edu.cn}\\
\texttt{xiangxiangnie@sdu.edu.cn},
\texttt{xzhang@xidian.edu.cn}
}
\date{}

\begin{document}
\maketitle

\begin{abstract}
We prove that every $n$-vertex simple planar graph with no copy of $C_8$
has at most
\[
  \frac{69}{25}(n-2)
\]
edges, for every $n\ge 8$.  This improves the best known bound
\[
  \frac{323}{108}n-6
  \qquad \text{for every } n\ge 27.
\]
\end{abstract}

{\small

\par}

\section{Introduction}

For a graph $H$, the \defterm{planar Tur\'an number} $\exP(n,H)$ is the largest
number of edges in an $n$-vertex simple planar graph containing no copy of
$H$ as a subgraph.  This is the planar analogue of the classical Tur\'an
problem~\cite{Turan1941}.  Since planar graphs have at most $3n-6$ edges,
the problem is to determine the correct linear coefficient and the extremal
constructions.  For background on planar Tur\'an problems we refer to the survey
of Lan, Shi and Song~\cite{LanShiSong2021}.

Planar Tur\'an numbers of cycles have been studied intensively since the work
of Dowden~\cite{Dowden2016}.  Euler's formula gives
\[
  \exP(n,C_3)=2n-4 \qquad(n\ge3).
\]
Dowden proved the following upper bounds for $C_4$ and $C_5$:
\[
  \exP(n,C_4)\le \frac{15}{7}(n-2)\qquad(n\ge4),
  \qquad
  \exP(n,C_5)\le \frac{12n-33}{5}\qquad(n\ge 11).
\]
For $C_6$, Lan, Shi and Song~\cite{LanShiSong2019} obtained
$\exP(n,C_6)\le 18(n-2)/7$ for $n\ge6$, and Ghosh, Gy\H{o}ri, Martin, Paulos and
Xiao~\cite{GhoshGyoriMartinPaulosXiao2022} proved the stronger bound
\[
  \exP(n,C_6)\le \frac52 n-7
  \qquad (n\ge 18).
\]
Results of Shi, Walsh and Yu~\cite{ShiWalshYuC7} and of Gy\H{o}ri, Li and
Zhou~\cite{GyoriLiZhou2023} address $C_7$; in particular, the former gives
\[
  \exP(n,C_7)\le \frac{18}{7}n-\frac{48}{7}
  \qquad (n\ge 39),
\]
with equality attained for infinitely many $n$.

For longer cycles the behaviour is more subtle.  Ghosh, Gy\H{o}ri, Martin,
Paulos and Xiao~\cite{GhoshGyoriMartinPaulosXiao2022} conjectured that, for
each $k\ge 7$ and all sufficiently large $n$,
\[
  \exP(n,C_k)\le
  \frac{3(k-1)}{k}n-\frac{6(k+1)}{k}.
\]
For $k=8$, the conjectured bound has leading coefficient $21/8$.
Cranston, Lidick\'y, Liu and Shantanam~\cite{CranstonLidickyLiuShantanam2022}
disproved the conjecture for every $k\ge 11$, but their counterexamples do
not settle the case $C_8$.  Their construction
also gives
\[
  \exP(n,C_8)\ge \frac{21}{8}n-\frac{27}{4}
\]
for infinitely many $n$.

For comparison, a general result of Shi, Walsh and
Yu~\cite{ShiWalshYuDense} gives the following bound for the \defterm{theta graph}
$\theta_k$, obtained from $C_k$ by adding a chord that forms a triangle
with two consecutive edges of the cycle, and gives
\[
  \exP(n,\theta_k)\le
  3n-6-\frac{n}{4k^{\log_2 3}}
  \qquad
  (k\ge 4,\ n\ge k^{\log_2 3}).
\]
Here we use the range in which that recursive estimate is applied.  Since a
copy of $\theta_k$ contains a copy of $C_k$, every $C_k$-free graph is
$\theta_k$-free, and the estimate also applies to $C_k$-free planar
graphs.  Taking $k=8$, and using
$8^{\log_2 3}=27$ and $4\cdot27=108$, gives
\[
  \exP(n,C_8)\le
  3n-6-\frac{n}{108}
  =
  \frac{323}{108}n-6
  \qquad (n\ge 27).
\]
Thus this general comparison bound has leading coefficient
$323/108\approx 2.991$, whereas our bound has coefficient
$69/25=2.76$.  The construction cited above has leading coefficient
$21/8=2.625$.
These two coefficients do not coincide.  Thus the theorem below is an improved
upper bound, not a determination of the exact planar Tur\'an number: we do not
claim that $69/25$ is best possible, that the displayed bound is
attained, or that equality cases are characterized.

We prove the following direct $C_8$-specific improvement.

\begin{theorem}\label{thm:turan}
For every $n\ge 8$,
\[
  \exP(n,C_8)\le \frac{69}{25}(n-2).
\]
\end{theorem}

This paper is organized as follows.  Section~\ref{sec:local-inputs} states
and proves the local structural inputs, including the finite verification
claims used by the proof.  Section~\ref{sec:main-proof} derives
Theorem~\ref{thm:turan} from these inputs.  The
\hyperref[sec:local-proofs]{Appendix} records the encodings used for these
finite enumerations, together with the algorithms, certificates, and audit
details that make the finite verifications reproducible.

\section{Local certificates and their consequences}\label{sec:local-inputs}

The discharging proof uses only a small number of local facts.  This section
states those facts in graph-theoretic form and proves the consequences needed
later.  The finite assertions below are certified by exhaustive enumerations of
finite plane patches around a specified root face or root triangle; the
encodings, certificate files, and reproducibility details are recorded in the
\hyperref[sec:local-proofs]{Appendix}.

Throughout the paper all graphs are finite and simple.  A \defterm{plane graph}
is a planar graph with a fixed embedding.  A \defterm{$k$-face} is a face of
degree $k$, and a \defterm{$4^+$-face} is a face of degree at least $4$.
A graph is \defterm{$C_8$-free} if it has no subgraph isomorphic to an
8-cycle; the cycle need not be induced.  For a plane graph $G$, its
\defterm{geometric dual} is denoted by $G^*$, and the \defterm{dual distance}
between faces $f$ and $g$ is denoted by $d^*(f,g)$.  In the main proof we
reduce to the 2-connected case, where every face boundary is a cycle.

\begin{definition}[\textit{Nearest faces and equal-split load}]\label{def:nearest-load}
Let $G$ be a 2-connected plane graph and let $g$ be a triangular face.
If $G$ has at least one $4^+$-face, the \defterm{nearest-face set} of
$g$ is
\[
  \Near(g)=\{f:\ f\text{ is a }4^+\text{-face and }d^*(f,g)
       \text{ is minimum among }4^+\text{-faces}\}.
\]
If no $4^+$-face exists, set $\Near(g)=\varnothing$.  For a $4^+$-face
$f$, a triangle $g$ is \defterm{$f$-bad} when
\[
  f\in \Near(g)\quad\text{and}\quad d^*(f,g)\le 3.
\]
The \defterm{nearest-bad count} of $f$ is
\[
  b(f)=|\{g:\ g\text{ is }f\text{-bad}\}|.
\]
The \defterm{equal-split load} of $f$ is
\[
  \ell(f)=\sum_{\substack{g\text{ a 3-face}\\ f\in \Near(g)}}
  \frac{1}{|\Near(g)|}.
\]
\end{definition}

Thus each triangular face carries one unit of load and splits it equally among
its nearest $4^+$-faces.  A triangle whose summand appears in $\ell(f)$ is a
\defterm{contributor to $f$}.

For a $4^+$-face $f$ and a triangular face $g$, define the
\defterm{shortest-entry set}
\[
  S_f(g)=\{e\in E(\partial f):\text{ some shortest dual path from $f$ to
  $g$ first leaves $f$ across $e$}\}.
\]
For a boundary edge $e$ of $f$, define
\[
  c_{f,e}(g)=
  \frac{\mathbf 1_{\{f\in\Near(g)\}}\mathbf 1_{\{e\in S_f(g)\}}}
       {|\Near(g)|\,|S_f(g)|}
\]
and
\[
  L_f(e)=\sum_g c_{f,e}(g).
\]
Then
\[
  \ell(f)=\sum_{e\in E(\partial f)}L_f(e),
\]
because each contribution assigned to $f$ is divided among the shortest
boundary entries of $f$ and then recombined.

\begin{example}[Equal splitting]
If a triangle $g$ has $\Near(g)=\{f_1,f_2\}$, then it contributes
$1/2$ to each of $\ell(f_1)$ and $\ell(f_2)$.  If instead
$\Near(g)=\{f\}$, then it contributes $1$ to $\ell(f)$.  Thus six
unique-nearest contributors to a $4$-face contribute $6$, while a further
contributor with at least three nearest $4^+$-faces contributes at most
$1/3$.
\end{example}

\begin{example}[Shortest entries]
Suppose $f\in\Near(g)$, $\Near(g)=\{f,h\}$, and the shortest dual paths
from $f$ to $g$ leave $f$ precisely through two boundary edges
$e_1,e_2$.  Then $S_f(g)=\{e_1,e_2\}$, so
\[
  c_{f,e_1}(g)=c_{f,e_2}(g)=\frac{1}{2\cdot2}=\frac14,
\]
and $c_{f,e}(g)=0$ for every other boundary edge $e$ of $f$.  The total
contribution of $g$ to $f$ is still $1/2$, now split equally between the
two shortest entries.
\end{example}

In the finite local assertions below, a \defterm{finite plane patch} means a
finite plane subgraph together with the faces already exposed.  Its boundary
edges are the edges incident with exactly one exposed face.  A \defterm{root}
face or root triangle is the distinguished face from which dual distances are
measured.

\begin{theorem}[Finite local certificates]\label{thm:finite-certificates}
The following finite local assertions hold.
\begin{enumerate}
\item[(C1)]\label{cert:C1} Let $P$ be a simple plane patch with a distinguished
root triangle. Assume that every face of $P$ at dual distance at most $3$
from the root triangle is triangular. Then $P$ is not contained in any simple
plane $C_8$-free graph that has a face at dual distance $4$ from the root
triangle, and $P$ has at most seven vertices.
\item[(C2)]\label{cert:C2} If a root $4^+$-face has degree $d\in\{4,5,6,7\}$, then the
maximum possible number of bad contributing triangles is
\[
\begin{array}{c|cccc}
 d&4&5&6&7\\
\hline
 \max b&7&7&5&4 .
\end{array}
\]
\item[(C3)]\label{cert:C3} For a root $4$-face, the number of contributors whose unique
nearest $4^+$-face is the root is at most $6$.
\item[(C4)]\label{cert:C4} Let $f$ be a root $4$-face. If exactly six contributors
have $f$ as their unique nearest $4^+$-face and there is one further
contributor, then that further contributor has at least two other nearest
$4^+$-faces, distinct from each other and from $f$.
\item[(C5)]\label{cert:C5} Let $f$ be a root face of degree at least $9$, and let
$e\in E(\partial f)$ be a distinguished boundary edge.  The finite local
upper bound for the load $L_f(e)$ is at most $9/2$.
\item[(C6)]\label{cert:C6} Among simple graphs on eight vertices with $17$ or $18$ edges,
there is no graph that is both planar and non-Hamiltonian.
\end{enumerate}
\end{theorem}

\begin{table}[t]
\caption{Use of the finite certificates in the proof.}
\label{tab:certificate-use}
\begin{center}
\small
\renewcommand{\arraystretch}{1.12}
\begin{tabular}{c p{0.46\linewidth} p{0.31\linewidth}}
\toprule
\textbf{Certificate} & \textbf{Certified local fact} & \textbf{Consequence used later}\\
\midrule
\CertOne & radius-three all-triangular neighbourhoods are not contained in a simple plane $C_8$-free graph with a face at dual distance four, and have at most seven vertices
& Lemma~\ref{lem:covering}\\
\CertTwo & small root faces have bad-count maxima $7,7,5,4$
& Lemma~\ref{lem:small-degree}\\
\CertThree & a $4$-face has at most six unique-nearest contributors
& Lemma~\ref{lem:four-load}\\
\CertFour & in the dangerous $4$-face case, the further contributor has two other nearest $4^+$-faces
& Lemma~\ref{lem:four-load}\\
\CertFive & the load $L_f(e)$ through one distinguished edge of a large face is at most $9/2$
& Lemma~\ref{lem:large-load}\\
\CertSix & no planar non-Hamiltonian $8$-vertex graph has $17$ or $18$ edges
& base case of Theorem~\ref{thm:turan}\\
\bottomrule
\end{tabular}
\end{center}
\end{table}

\begin{proposition}[Lifting and domination principle]\label{prop:lifting-domination}
The finite certificates in Theorem~\ref{thm:finite-certificates} apply to every
local configuration arising in a simple 2-connected $C_8$-free plane graph
with minimum degree at least $3$.  More precisely, every triangular
neighbourhood around the specified root needed for \CertOne--\CertFour{} occurs in the
corresponding finite search,
and the finite assertion in \CertFive{} gives an upper bound for the true load
$L_f(e)$ for every boundary edge $e$ of every face $f$ with $d(f)\ge9$.
\end{proposition}

\begin{proof}
For \CertOne--\CertFour, expose triangular faces outward from the chosen root face or
root triangle.  At any stage, consider an edge $uv$ incident with exactly one
already exposed triangular face.  If the face on the other side of $uv$ is
also triangular, then its third vertex is either new to the exposed patch or
already lies on the same boundary cycle of the exposed patch as $uv$.  These
are exactly the two planar ways a triangular face can be added.  Exposing faces
in nondecreasing true dual distance from the root gives an enumerated
representative with the required dual-distance layers.  Any extra patch kept by
the finite search only enlarges the search space, which is harmless because the
certificates give upper bounds.

The only local rejections used in these finite searches are necessary under the
hypotheses: repeated facial triangles, a third facial incidence on one edge,
non-plane boundary order, and an actual $8$-cycle cannot occur in a simple
plane $C_8$-free graph.  When the search excludes two facial triangles
sharing a two-edge boundary path, the internal vertex of that path would have
degree $2$ in the ambient graph; this is impossible under $\delta(G)\ge3$.
Hence no genuine rooted neighbourhood required by \CertOne--\CertFour{} is lost.

For \CertFive, keep a boundary edge $e$ of $f$ and the two
neighbouring boundary edges visible.  The rest of $\partial f$ is used only
through its cyclic order.  A shortest path from $f$ to a triangle
contributing through $e$ begins across $e$, so every contributing triangle
is reached by successively exposing genuine triangular faces from this visible
three-edge part of $\partial f$.  If an edge incident with exactly one
exposed face can affect the contribution through $e$, then the face on its
other side is either triangular, the root face again, or a non-root
$4^+$-face; these are precisely the alternatives covered by the finite
assertion in \CertFive.  Restoring the omitted part of the root boundary can only
add shortest entries, create a shorter root route through another entry, or add
further non-root competitors.  Each of these changes enlarges a denominator or
makes the contribution through $e$ vanish.  Thus the certified quantity in
\CertFive{} dominates the true sum $L_f(e)$.
\end{proof}

\begin{lemma}[Covering by nearby $4^+$-faces]\label{lem:covering}
Let $G$ be a simple 2-connected $C_8$-free plane graph with
$\delta(G)\ge 3$ and $|V(G)|\ge 8$.  Then every 3-face $g$ of $G$
has a $4^+$-face at dual distance at most $3$.  In particular,
$\Near(g)\ne\emptyset$, and $g$ is $f$-bad for every $f\in\Near(g)$.
\end{lemma}

\begin{proof}
Suppose that a triangular face $g$ has no $4^+$-face within dual distance
$3$.  Then the radius-three neighbourhood of $g$ consists only of
triangular faces.  By Proposition~\ref{prop:lifting-domination} this
neighbourhood is represented in the finite search certified in \CertOne, and \CertOne{}
says that this neighbourhood is not contained in any simple plane
$C_8$-free graph with a face at dual distance $4$ and that the neighbourhood
has at most seven vertices.  A vertex of $G$ outside this neighbourhood would
create such an additional face.  Hence
$|V(G)|\le7$, a contradiction.  The remaining assertions follow from the
definitions.
\end{proof}

\begin{lemma}[Small root-face bounds]\label{lem:small-degree}
Let $G$ be a simple 2-connected $C_8$-free plane graph with
$\delta(G)\ge 3$ and $|V(G)|\ge 8$, and let $f$ be a $4^+$-face of degree
$d\in\{4,5,6,7\}$.  Then
\[
\begin{array}{c|cccc}
d&4&5&6&7\\
\hline
b(f)&\le 7&\le 7&\le 5&\le 4.
\end{array}
\]
\end{lemma}

\begin{proof}
By Lemma~\ref{lem:covering}, only triangles at dual distance at most $3$ from
$f$ can be $f$-bad.  To certify that such a triangle is truly $f$-bad, it
is enough to resolve the triangular witnesses that could lead to a closer
$4^+$-face.  These witnesses lie within the finite neighbourhood used
in \CertTwo.  Proposition~\ref{prop:lifting-domination} transfers every genuine
neighbourhood to the finite search certified in \CertTwo, and \CertTwo{} gives the displayed
maxima.
\end{proof}

\begin{lemma}[Refined four-face equal-split load]\label{lem:four-load}
If $G$ satisfies the hypotheses above and $f$ is a 4-face, then
\[
  \ell(f)\le \frac{19}{3}.
\]
\end{lemma}

\begin{proof}
By Lemma~\ref{lem:small-degree}, the face $f$ has at most seven contributors.
Let $u$ be the number of contributors whose unique nearest $4^+$-face is
$f$.  Certificate \CertThree, together with Proposition~\ref{prop:lifting-domination},
gives $u\le6$.

If $u\le5$, then the unique-nearest contributors add at most $5$, and each
remaining contributor has at least two nearest $4^+$-faces and contributes at
most $1/2$.  Hence
\[
  \ell(f)\le 5+\frac{7-5}{2}=6.
\]
If $u=6$ and there is no seventh contributor, then $\ell(f)=6$.  If a
seventh contributor exists, \CertFour{} gives two other nearest $4^+$-faces for
that triangle, distinct from each other and from $f$.  Its contribution to
$\ell(f)$ is therefore at most $1/3$, and
\[
  \ell(f)\le 6+\frac13=\frac{19}{3}.
\]
\end{proof}

\begin{lemma}[Large-face load]\label{lem:large-load}
Let $G$ be a simple $2$-connected $C_8$-free plane graph with
$\delta(G)\ge3$ and $|V(G)|\ge8$.  If $f$ is a face of degree at least
$9$, then
\[
  \ell(f)\le \frac92 d(f).
\]
\end{lemma}

\begin{proof}
By \CertFive{} and Proposition~\ref{prop:lifting-domination}, every boundary edge
$e\in E(\partial f)$ satisfies $L_f(e)\le9/2$.  Summing over all boundary
edges and using $\ell(f)=\sum_{e\in E(\partial f)}L_f(e)$ gives
\[
  \ell(f)\le \sum_{e\in E(\partial f)}\frac92=\frac92 d(f).
\]
\end{proof}

\begin{lemma}[Eight-vertex base case]\label{lem:n8-hamiltonian}
Every simple planar graph on eight vertices with $17$ or $18$ edges is
Hamiltonian.
\end{lemma}

\begin{proof}
This is exactly \CertSix.  On eight vertices, a Hamiltonian cycle is a copy of
$C_8$, so the certificate rules out every planar $C_8$-free graph with
$17$ or $18$ edges.
\end{proof}

\section{Proof of the main theorem}\label{sec:main-proof}

The local inputs and the mathematical reductions needed for discharging were
proved in Section~\ref{sec:local-inputs}.  We now use them to prove the main
theorem.
\begin{proof}[Proof of Theorem~\ref{thm:turan}]
Put $c=69/25$.  We prove by induction on $n$ that every $n$-vertex
simple $C_8$-free planar graph has at most $c(n-2)$ edges, for $n\ge8$.
For $n=8$, $c(8-2)=414/25<17$, while every eight-vertex simple planar graph
has at most $18$ edges.  Lemma~\ref{lem:n8-hamiltonian} therefore implies
that any eight-vertex planar graph with more than $c(8-2)$ edges is
Hamiltonian, and hence contains a copy of $C_8$.

Let $G$ be a counterexample with the minimum possible number $n$ of
vertices.  Every subgraph used below remains simple, planar, and $C_8$-free.
For a smaller graph $H$ of order $m$, use induction if $m\ge8$, and use
the elementary bound $p(1)=0$, $p(2)=1$, $p(m)=3m-6$ for $3\le m\le7$.
In the latter range $p(m)\le c(m-1)$.

If $G$ is disconnected, with component orders $n_1,\ldots,n_t$, then
\[
  e(G)\le c\sum_i(n_i-1)=c(n-t)\le c(n-2),
\]
a contradiction.  If $G$ has a cut vertex, write
$G=G_1\cup G_2$, $G_1\cap G_2=\{x\}$, and $|V(G_i)|=n_i$, so
$n_1+n_2=n+1$.  If both $n_i\ge8$, induction gives
$e(G)\le c(n_1-2)+c(n_2-2)=c(n-3)<c(n-2)$.  If
$n_1\le7\le n_2$, then $e(G)\le c(n_1-1)+c(n_2-2)=c(n-2)$.  If both
$n_i\le7$, then $p(n_1)+p(n_2)\le c(n_1+n_2-3)=c(n-2)$: for
$n_1,n_2\ge3$ this is $3(n_1+n_2)-12\le c(n_1+n_2-3)$, with
$n_1+n_2\le14$; if one side has order $2$, it reduces to
$3m-5\le c(m-1)$ for $3\le m\le7$; and if both have order $2$, it is
$2\le c$.  Thus $G$ is 2-connected.

If $G$ has a vertex $v$ of degree at most $2$, then induction gives
$e(G)\le c(n-3)+2\le c(n-2)$ when $n-1\ge8$, since $c>2$; the remaining
case is $n=8$, already covered.  Hence $\delta(G)\ge3$.  We may therefore
assume that $G$ is simple, 2-connected, $C_8$-free, $\delta(G)\ge3$, and
$n\ge9$, so Lemma~\ref{lem:covering} applies.

Fix a plane embedding, and let $F$ be the set of faces.  Give each face
initial charge $\mu(f)=2d(f)-4$.  Euler's formula gives
\[
  \sum_{f\in F}\mu(f)=2\sum_{f\in F}d(f)-4|F|=4e(G)-4|F|=4n-8.
\]
Set $\alpha=50/69$.  Each face sends charge $\alpha$ to each incident
vertex.  Since $G$ is 2-connected, each face boundary is a cycle, so the
total charge sent to vertices is $2\alpha e(G)$, and the residual face charge
is $\mu^*(f)=(2-\alpha)d(f)-4$.

Now each triangular face $g$ receives its required compensation equally from
the faces in $\Near(g)$: every $4^+$-face $f\in\Near(g)$ sends
$(3\alpha-2)/|\Near(g)|=(4/23)/|\Near(g)|$ to $g$.  Since
$3(2-\alpha)-4=2-3\alpha=-4/23$, every triangular face has final charge
$0$.

Let $f$ be a $4^+$-face of degree $d$, and write $\mu^\times(f)$ for
its final face charge.  From the second discharging step,
\[
  \mu^\times(f)=(2-\alpha)d-4-\frac{4}{23}\ell(f)
  =\frac{88}{69}d-4-\frac{4}{23}\ell(f).
\]
If $d=4$, Lemma~\ref{lem:four-load} gives
$\mu^\times(f)\ge (88/69)4-4-(4/23)(19/3)=0$.  For $d=5,6,7$,
Lemma~\ref{lem:small-degree} gives $\ell(f)\le b(f)\le 7,5,4$, respectively;
the first inequality follows from Lemma~\ref{lem:covering}.  Thus the
corresponding lower bounds for $\mu^\times(f)$ are $80/69$, $192/69$,
and $292/69$.  A face of degree $8$ cannot occur, because its boundary
would be an 8-cycle.  Finally, if $d\ge9$, Lemma~\ref{lem:large-load} gives
$\ell(f)\le9d/2$, and hence
\[
  \mu^\times(f)\ge \frac{88}{69}d-4-\frac{4}{23}\cdot\frac{9d}{2}
  =\frac{34d}{69}-4\ge \frac{30}{69}>0.
\]
Thus every face has nonnegative final face charge.

Charge is conserved, so $4n-8=2\alpha e(G)+\sum_{f\in F}\mu^\times(f)\ge
2\alpha e(G)$.  Therefore
\[
  e(G)\le \frac{2}{\alpha}(n-2)=\frac{69}{25}(n-2),
\]
contradicting the choice of $G$.  The theorem follows.
\end{proof}

\section*{Declaration on AI-assisted tools}
OpenAI Codex was used only as an assistive tool for implementing and checking
computer programs used in the finite computations.

\section*{Data and code availability}
The source codes, recorded outputs, machine-readable certificates, and
reproduction scripts are available on the homepage of the corresponding
author: \url{https://web.xidian.edu.cn/zhangxin/papers.html}.

\bibliographystyle{plainurl}
\begingroup
\hbadness=10000
\IfFileExists{c8_planar_turan_ref.bib}
  {\bibliography{c8_planar_turan_ref}}
  {\bibliography{manuscript/c8_planar_turan_ref}}

\begin{thebibliography}{10}

\bibitem{CranstonLidickyLiuShantanam2022}
D.~W. Cranston, B.~Lidick{\'y}, X.~Liu, and A.~Shantanam.
\newblock Planar {Tur{\'a}n} numbers of cycles: a counterexample.
\newblock {\em Electron. J. Combin.}, 29(3):P3.31, 2022.
\newblock \href {http://dx.doi.org/10.37236/10774} {\path{doi:10.37236/10774}}.

\bibitem{Dowden2016}
C.~Dowden.
\newblock Extremal {$C_4$}-free/{$C_5$}-free planar graphs.
\newblock {\em J. Graph Theory}, 83(3):213--230, 2016.
\newblock \href {http://dx.doi.org/10.1002/jgt.21991}
  {\path{doi:10.1002/jgt.21991}}.

\bibitem{GhoshGyoriMartinPaulosXiao2022}
D.~Ghosh, E.~Gy{\H{o}}ri, R.~R. Martin, A.~Paulos, and C.~Xiao.
\newblock Planar {Tur{\'a}n} number of the 6-cycle.
\newblock {\em SIAM J. Discrete Math.}, 36(3):2028--2050, 2022.
\newblock \href {http://dx.doi.org/10.1137/21M140657X}
  {\path{doi:10.1137/21M140657X}}.

\bibitem{GyoriLiZhou2023}
E.~Gy{\H{o}}ri, A.~Li, and R.~Zhou.
\newblock The planar {Tur{\'a}n} number of the seven-cycle.
\newblock arXiv:2307.06909v2, 2023.
\newblock \href {http://arxiv.org/abs/2307.06909v2}
  {\path{arXiv:2307.06909v2}}, \href
  {http://dx.doi.org/10.48550/arXiv.2307.06909}
  {\path{doi:10.48550/arXiv.2307.06909}}.

\bibitem{LanShiSong2019}
Y.~Lan, Y.~Shi, and Z.-X. Song.
\newblock Extremal theta-free planar graphs.
\newblock {\em Discrete Math.}, 342(12):111610, 2019.
\newblock \href {http://dx.doi.org/10.1016/j.disc.2019.111610}
  {\path{doi:10.1016/j.disc.2019.111610}}.

\bibitem{LanShiSong2021}
Y.~Lan, Y.~Shi, and Z.-X. Song.
\newblock Planar {Tur{\'a}n} number and planar anti-{R}amsey number of graphs.
\newblock {\em Oper. Res. Trans.}, 25(3):201--216, 2021.
\newblock \href {http://dx.doi.org/10.15960/j.cnki.issn.1007-6093.2021.03.013}
  {\path{doi:10.15960/j.cnki.issn.1007-6093.2021.03.013}}.

\bibitem{ShiWalshYuDense}
R.~Shi, Z.~Walsh, and X.~Yu.
\newblock Dense circuit graphs and the planar {Tur{\'a}n} number of a cycle.
\newblock {\em J. Graph Theory}, 108(1):27--38, 2025.
\newblock \href {http://dx.doi.org/10.1002/jgt.23165}
  {\path{doi:10.1002/jgt.23165}}.

\bibitem{ShiWalshYuC7}
R.~Shi, Z.~Walsh, and X.~Yu.
\newblock Planar {Tur{\'a}n} number of the 7-cycle.
\newblock {\em European J. Combin.}, 126:104134, 2025.
\newblock \href {http://dx.doi.org/10.1016/j.ejc.2025.104134}
  {\path{doi:10.1016/j.ejc.2025.104134}}.

\bibitem{Turan1941}
P.~Tur{\'a}n.
\newblock On an extremal problem in graph theory.
\newblock {\em Mat. Fiz. Lapok}, 48:436--452, 1941.

\bibitem{Wagner1937}
K.~Wagner.
\newblock {\"U}ber eine {Eigenschaft} der ebenen {Komplexe}.
\newblock {\em Math. Ann.}, 114:570--590, 1937.
\newblock \href {http://dx.doi.org/10.1007/BF01594196}
  {\path{doi:10.1007/BF01594196}}.

\end{thebibliography}
\endgroup

\appendix

\phantomsection
\section*{Appendix. Verification details for the local inputs}\label{sec:local-proofs}
\addcontentsline{toc}{section}{Appendix. Verification details for the local inputs}
\setcounter{subsection}{0}
\renewcommand{\thesubsection}{A\arabic{subsection}}

This appendix gives the verification details for the six finite local
certificates \CertOne--\CertSix{} stated in
Theorem~\ref{thm:finite-certificates}.  For each search it specifies the finite
search data, the allowed extensions, and the recorded output from which the
corresponding local bound follows.  The local lemmas are proved in the main
text; the material here supplies the reproducible implementation and audit
details.  The complete verification suite took $854.7$ seconds, about
$14.3$ minutes, on the hardware specified for the computation.

\subsection{Search states and rooted patches}

This subsection fixes the common finite-patch language used in the
verifications of certificates \CertOne--\CertSix.  The subsequent subsections
apply this language to the six local certificates stated in
Theorem~\ref{thm:finite-certificates}.

We now give the formal data stored by the finite searches.  The root is the
distinguished face specified in the relevant certificate.  A layer is a
dual-distance layer measured from that root.  The terms open edge, closed edge,
and virtual side are defined in this subsection.  Each unfilled complementary
region is represented by its boundary in cyclic order.  Ordinary boundaries are
simple cycles, so a listed boundary vertex occurs only once on its component.
Root edges are not open witness edges for a competing $4^+$-face, because the
root face itself already occupies one side.

\begin{definition}[\textit{Rooted search state}]
A \defterm{rooted search state} consists of a listed plane patch $P$, one
distinguished root face, and one boundary record for each unfilled
complementary region of $P$.  An \defterm{ordinary boundary component} is a
cyclic list of at least three distinct vertices.  In the root-edge model, parts
of the omitted portion of the root boundary, called the \defterm{omitted root
arc}, may instead be recorded as virtual sides.  When a newly encountered ambient vertex lies on such a
side, the virtual-anchor branch subdivides it into two ordered virtual
subarcs.  If a real edge and one virtual subarc enclose a sector with only
their two endpoints currently listed, the search retains a
\defterm{compressed real--virtual component}; it does not discard that
unfilled sector.

Each listed triangle carries an \defterm{exposure layer}, assigned through the
parent edge by which the triangle is added.  For the ordinary searches, a
genuine patch has a retained breadth-first exposure in which this label equals
the true ambient dual distance from the root.  Other retained exposure orders
may carry larger labels and merely enlarge the search.  In the sealed search
the stored value is the construction layer.  New adjacencies and $R$-seals
may later shorten the sealed root distance, so candidate and load decisions
use a fresh sealed-dual breadth-first computation.  Witness exposure uses a
separate breadth-first computation in the root-and-triangles incidence graph,
which records exactly the triangular interior of a path to its first
competing $4^+$-face.

A real boundary edge is open when it is incident with exactly one listed face
and closed when it is incident with two.  An open edge carrying an $H$- or
$R$-seal is resolved, while an $O$-seal does not resolve an edge.  A new
triangle is attached only across an unresolved open real edge; its third
vertex is either new or already lies on the same boundary component.  A
\defterm{state key} is a deterministic record
used to suppress duplicate descriptions.  Two states are merged only when
the key shows that they have the same available boundary extensions and the
  same counting data.  Keeping additional isomorphic copies is allowed.

The complete state record is the tuple
\[
 (V_P,E_P,\mathcal T,\rho,\mathcal B,\mathcal V,\sigma).
\]
Here $\mathcal T$ is the list of oriented facial triangles; $\rho$ is the
root record and all layer data; $\mathcal B$ is the ordered collection of
boundary records; $\mathcal V$ gives the ordered virtual subarcs and their
endpoints; and $\sigma$ records the $H/R/O$-seals on real sides.  The
parent edge used when a layer is assigned is transition data: it determines
the child layer and is verified as part of successor legality, but it is not
an additional stored field of a terminal sealed state.  A transition chooses
one unresolved real boundary edge and either attaches a triangle using a new
vertex, an old vertex on the same component, or a virtual anchor on its unique
virtual subarc, or records its true second side by an $H$- or $R$-seal.
The two neighbouring displayed root edges may receive the nonblocking
$O$-marker.  A child is retained exactly when the explicit rejection rules
in Proposition~\ref{prop:search-state-domination} do not apply.
\end{definition}

\paragraph*{Relevance predicate.}
The exact deterministic predicate is as follows.  Construct $D_S$, its root
distances $r_i$, merged-$H$ distances $h_i$, and displayed-entry counts
$p_i$.  Put
\[
 C(S)=\{(i,r_i):r_i\le3,\ h_i\ge r_i,\ p_i>0\},
\]
as a set of eligible listed triangles.  Let $O(S)$ be the set of real
boundary edges having one listed incidence and carrying no $H$- or
$R$-seal, keyed by their sorted endpoint pairs.  If
$C(S)=\varnothing$, the selected edge is the
lexicographically least member of $O(S)$ among the displayed root-window
edges $30,01,12$; if none exists, the state is terminal.  Otherwise, for
each $(i,k)\in C(S)$, insert into an eligible-edge set every edge of every
listed triangle $T_j$ whose distance from $T_i$ in the
root-and-listed-triangles incidence graph is at most $k-1$, provided that
edge belongs to $O(S)$.  The selected edge is the lexicographically least
edge in this eligible-edge set.  If the set is empty, the state is terminal.
Construction layer is not used in selecting the edge.

\begin{lemma}[Faithful cut-open root-arc model]
\label{lem:faithful-cut-open}
Fix a face $f$, a boundary edge $e$, and a genuine local configuration
around them.  Retain $e$ and its two neighbouring root edges.  Let $Q$ be
the remaining arc of $\partial f$, oriented from vertex $2$ to vertex $3$.
The sealed search has a cut-open representation of the configuration with the
following properties.
\begin{enumerate}
\item Listed vertices map injectively to ambient vertices; every displayed
real edge, listed triangular face, and shared-edge dual adjacency maps to the
corresponding genuine object.
\item The unlisted portions of $Q$ are represented by ordered
\defterm{virtual subarcs}.  Their interiors are pairwise disjoint and contain
no listed vertex.  If a true triangular continuation has a new third vertex
in the interior of a virtual subarc, the search has a
\defterm{virtual-anchor branch}: it inserts that vertex into the subarc and
then performs the ordinary old-boundary-vertex split.  Otherwise the search
uses the ordinary new-vertex or old-vertex branch.  Consequently an ambient
vertex is never introduced twice.
\item Virtual subarcs are never treated as graph edges.  Consequently every
$8$-cycle found by the search consists of eight distinct genuine vertices
and real edges and is an $8$-cycle of the ambient graph.
\item Every triangle whose true contribution through $e$ is nonzero is
listed.  Restoring the still unlisted parts of $Q$ and the faces incident
with them may add root entries, a shorter root route, or additional non-root
$4^+$-competitors.  These changes can only enlarge a true denominator or
make the contribution through $e$ zero.
\end{enumerate}
\end{lemma}

\begin{proof}
Delete the interior of $f$ from the sphere.  The remaining surface is a
closed disk $D$ with boundary $\partial f$.  The three displayed root edges
are fixed on $\partial D$.  Initially the interior of $Q$ contains no
listed vertex and is represented by the single virtual subarc $23$.

We maintain the following lifting invariant.  Each listed vertex has one
ambient image, the vertex map is injective, and all listed real incidences are
genuine.  Each virtual side is assigned a closed subarc of $Q$; the
interiors of these assigned subarcs are pairwise disjoint, their union is the
unlisted part of $Q$, and their endpoints occur in the order recorded by the
boundary lists.  Every genuine unfilled sector is contained in an unfilled
sector represented by one search boundary component.  Every occurrence of a
listed vertex on the genuine sector boundary occurs on that search boundary,
and every unlisted portion of $Q$ on the genuine boundary is contained in
the subarc assigned to a virtual side of that same search boundary.  Search
sectors may be coarser than genuine sectors, which only admits extra
extensions.

Induct on the exposed triangular faces.  Suppose that a true triangle lies
across a displayed real open edge $uv$, and let $w$ be its third vertex.
If $w$ is already listed, the triangle lies in the sector incident with
$uv$, so the corresponding boundary component contains the displayed
occurrence of $w$; the old-vertex branch gives the true split.  If $w$ is
not listed and does not lie in the interior of an assigned virtual subarc, the
ordinary new-vertex branch introduces it.  Finally suppose that $w$ lies in
the interior of a virtual subarc.  That subarc is unique by the invariant.
The virtual-anchor branch subdivides it at $w$, then regards $w$ as an old
boundary vertex when the triangle is attached.  The two resulting boundary
components both contain the required endpoint occurrence of $w$, while the
two new virtual sides represent the two subarcs of $Q$ cut at $w$.
Thus the assigned virtual subarcs remain disjoint and ordered, and the ambient
vertex $w$ receives exactly one displayed label.  This proves the lifting
invariant for every triangular transition.  An $H$- or $R$-transition only
records the genuine second facial side of a real edge and does not change the
vertex or subarc data.

Because the vertex map is injective and virtual sides are excluded from the
cycle test, every rejected $8$-cycle is a genuine simple $8$-cycle.  We
first justify that every positive contributor through $e$ is listed.  Let
\[
  f,T_1,\ldots,T_k=g
\]
be a shortest dual path to such a contributor, with its first step through
$e$.  Every prefix is shortest, so $d^*(f,T_i)=i$ and
$e\in S_f(T_i)$.  Moreover, $f\in\Near(T_i)$: if a $4^+$-face $h$
were closer to $T_i$ than $f$, then a shortest $h$-to-$T_i$ path
followed by the suffix $T_i,\ldots,T_k$ would make $h$ closer to $g$
than $f$.  In particular, every $T_i$ is triangular and is itself an
eligible candidate through $e$.  The search first resolves the displayed
root-window edges.  Once $T_i$ is listed, the edge leading to
$T_{i+1}$ is an unresolved relevant edge incident with that candidate and
is eventually resolved by the deterministic rule.  Induction on $i$
therefore lists the whole path, including $g$.

It remains to compare loads.  The witness exposure lists or seals every side
that can provide a closer or tied $4^+$-face.  Restoring an unlisted root
edge on $Q$ supplies an additional route from the same root face.  If that
route is shorter, $e\notin S_f(g)$; if it is tied, it adds an entry to
$S_f(g)$.  A restored non-root $4^+$-face can only add a nearest competitor
or show that $f$ is not nearest.  Restored triangular faces cannot create an
unlisted positive contribution through $e$: any triangle with
$e\in S_f(g)$ is already listed along a shortest path beginning with $e$.
Thus every restored object either leaves the summand unchanged, enlarges one
of its denominators, or makes it zero.  This proves all four assertions.
\end{proof}
\begin{proposition}[Coverage of genuine patches]\label{prop:search-state-domination}
The ordinary searches contain a rooted plane-isomorphic or labelled copy of
every genuine rooted triangular subpatch of a simple $2$-connected
$C_8$-free plane graph with $n\ge8$ that satisfies their root and radius
conditions.  For the searches using the shared-path rejection, assume also
$\delta(G)\ge3$.  The sealed search contains, instead, the faithful cut-open
representation supplied by Lemma~\ref{lem:faithful-cut-open}; no plane
isomorphism between a sealed state and the uncompressed ambient patch is
claimed.

More precisely, if a genuine final patch can be exposed by successively adding
triangular faces, then at least one breadth-first exposure is followed by the
corresponding ordinary search, and the corresponding cut-open exposure is
followed by the sealed search.  Every rejection rule is either a necessary
condition under the stated hypotheses or an exposure convention satisfied by
one such representative.  The deterministic state keys do not identify two
states with different available completions.
\end{proposition}

\begin{proof}
For the ordinary searches, argue by induction on the number of listed
triangular faces.  The initial patch is the root boundary cycle and is
$2$-connected.  Adding a new third vertex across an edge gives that vertex
two distinct neighbours in the old patch, so the enlarged patch remains
$2$-connected; adding a triangle whose third vertex is old only adds edges
and also preserves $2$-connectedness.  Hence every unfilled complementary
region of an ordinary patch has a simple boundary cycle.  If a true triangle
lies across a real open boundary edge $uv$, its third vertex lies in the
closure of the incident complementary region.  It is therefore either new to
the patch or occurs exactly once on that region's boundary cycle.  The
new-vertex and old-vertex branches are precisely these two cases, and an
old-vertex branch splits one disk region along the two new edges into two disk
regions.  This also proves the ordinary boundary invariant used in
Lemma~\ref{lem:same-boundary}.  The sealed induction, including virtual
anchors and compressed real--virtual components, is
Lemma~\ref{lem:faithful-cut-open}.

The local rejection rules excluding repeated triangular faces, three listed
incidences on one edge, non-plane boundary order, and an $8$-cycle are
necessary for the simple plane $C_8$-free graphs with $n\ge8$ considered
here.  The searches other than the closed radius-three search also reject a
new triangle that would share a two-edge facial path with an already listed
face, including the root.  By Lemma~\ref{lem:shared-face-path}, such a shared
path would force its internal vertex to have degree $2$; hence this
rejection is valid under $\delta(G)\ge3$.  The closed radius-three search
does not use this rejection and needs no minimum-degree hypothesis.
The optional sparsity rejection in the closed
triangular-neighbourhood search is also necessary for a simple planar graph.
For an ordinary genuine patch, expose faces in nondecreasing true dual
distance from the root.  The resulting representative has correct layer
labels, so an ordinary layer bound is imposed only beyond the radius needed by
the corresponding lemma; any additional retained exposure histories are
overcounts.  In the sealed search, the stored construction layer is audited
separately: the finite certificate in
Lemma~\ref{lem:finite-guarded-sealed-certificate} verifies that no legal
relevant triangular branch is rejected by the construction-layer bound.
Finally, in the small-root and refined $4$-face searches the
minimum-parent-layer rule does not delete a rooted patch; it chooses a
breadth-first exposure order.  To see this formally, order the triangular
faces of a genuine rooted patch first by their true dual distance from the
root and then arbitrarily inside each layer.  When the next face $T$ in this
order is exposed, $T$ has at least one already exposed neighbour in the
previous layer along a shortest dual path from the root to $T$, and no
already exposed neighbour has smaller layer than that.  Therefore exposing
$T$ through a minimum-layer parent gives $T$ its true dual layer and
produces the same final patch.  Inducting over this order gives an exposure
sequence retained by the search.

It remains only to discuss deduplication.  The labelled-key programs do not
identify different labelled patches at all; they only put the stored lists in
a deterministic order.  For the canonical-key programs, the keys record the
root, the listed facial cycles with their layers, the boundary components,
and, when present, the seal data.  Thus two states with the same key have the
same real open boundary edges, the same admissible old and new third-vertex
extensions, the same future rejection tests, and the same counting data.
Keeping one representative therefore cannot remove a possible completion.
The implementation records the exact state invariants and keys for each program.
For completeness, key equality has the following successor-preservation
property.  Equality of rooted boundary keys supplies a root-preserving
bijection preserving the triangle list and layers, cyclic boundary order,
real/virtual side type, next unused label, and seals.  Any legal transition of
one representative is therefore transported to a legal transition of the same
type of the other, and the children again have equal normalized keys; the
converse follows from the inverse bijection.  The sealed program uses the same
principle with an order-independent selected-edge rule: among all eligible
unresolved open edges it chooses the lexicographically least unoriented edge.
Since the sealed key preserves the set of real open endpoint pairs, the
face--layer multiset, and all seal data, equal-key representatives have the
same eligible-edge set and hence the same selected edge.  Terminal status,
successor-key sets, and counting data are therefore preserved under sealed-key
deduplication.  For labelled keys the bijection is the identity.  The only
isomorphism key is used by the closed-neighbourhood search, whose transitions
and terminal tests depend only on the triangular incidences, layers, open
edges, $8$-cycle test, and planar sparsity retained by that key.  A hash
value or regression fingerprint is never used as a mathematical state
invariant.
\end{proof}

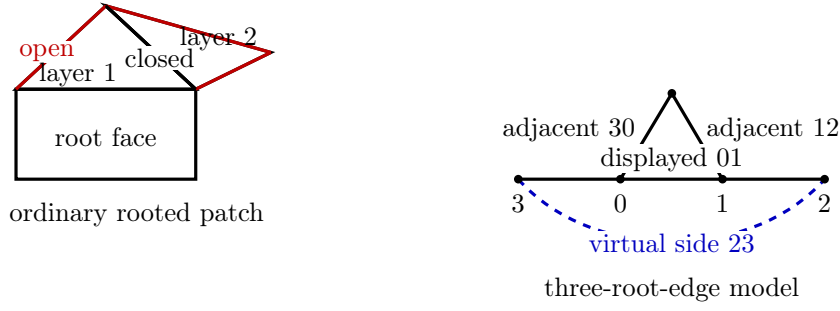
\begin{figure}[t]
\centering
\begin{tikzpicture}[scale=1.08, every node/.style={font=\small}]
  \begin{scope}
    \coordinate (a) at (0,0);
    \coordinate (b) at (2.2,0);
    \coordinate (c) at (2.2,1.1);
    \coordinate (d) at (0,1.1);
    \coordinate (e) at (1.1,2.12);
    \coordinate (u) at (3.12,1.55);

    \draw[very thick] (a)--(b)--(c)--(d)--cycle;
    \node at (1.1,0.52) {root face};

    \draw[very thick] (d)--(c)--(e)--cycle;
    \draw[very thick] (c)--(e)--(u)--cycle;

    \draw[very thick,red!75!black] (d)--(e);
    \draw[very thick,red!75!black] (e)--(u);
    \draw[very thick,red!75!black] (u)--(c);
    \draw[very thick] (c)--(e);

    \node[red!75!black,fill=white,inner sep=1pt] at (0.38,1.55) {open};
    \node[fill=white,inner sep=1pt] at (1.76,1.48) {closed};
    \node at (0.76,1.28) {layer $1$};
    \node at (2.48,1.72) {layer $2$};
    \node[align=center] at (1.48,-0.43) {ordinary rooted patch};
  \end{scope}

  \begin{scope}[xshift=6.15cm]
    \coordinate (p3) at (0,0);
    \coordinate (p0) at (1.25,0);
    \coordinate (p1) at (2.5,0);
    \coordinate (p2) at (3.75,0);
    \coordinate (pa) at (1.88,1.05);

    \draw[very thick] (p3)--(p0)--(p1)--(p2);
    \draw[very thick] (p0)--(pa)--(p1);
    \draw[very thick,blue!75!black,dashed] (p3) .. controls (0.9,-1.02) and (2.85,-1.02) .. (p2);

    \fill (p3) circle (1.5pt) node[below=2pt] {$3$};
    \fill (p0) circle (1.5pt) node[below=2pt] {$0$};
    \fill (p1) circle (1.5pt) node[below=2pt] {$1$};
    \fill (p2) circle (1.5pt) node[below=2pt] {$2$};
    \fill (pa) circle (1.5pt);

    \node[above=2pt,fill=white,inner sep=1pt] at (1.88,0) {displayed $01$};
    \node[above=13pt,fill=white,inner sep=1pt] at (0.62,0) {adjacent $30$};
    \node[above=13pt,fill=white,inner sep=1pt] at (3.12,0) {adjacent $12$};
    \node[blue!75!black,fill=white,inner sep=1pt] at (1.88,-0.78) {virtual side $23$};
    \node[align=center] at (1.88,-1.35) {three-root-edge model};
  \end{scope}
\end{tikzpicture}
\caption{A schematic rooted search state.  Left: red edges are open boundary
edges, the thick black non-root edge is closed, and the displayed triangles
have layers measured from the root face.  Right: in the three-root-edge model
the vertices $3,0,1,2$ lie on the displayed part of the root boundary, while
the blue dashed arc from $3$ to $2$ is a bookkeeping side for the omitted
part of the root $4^+$-face; it is not a graph edge.}
\label{fig:rooted-state}
\end{figure}

The boundary invariant follows by induction.  Initially the displayed
boundary of the sealed three-root-edge model is $3,0,1,2$, with $23$
virtual.  An ordinary new third vertex subdivides a real boundary edge, and an
old third vertex splits the current component.  If a new ambient third vertex
lies on the omitted root arc, the virtual-anchor branch first inserts it into
each candidate virtual subarc in turn and then performs the old-vertex split.
The branch using the subarc that contains the true vertex preserves its
position on the cut-open root boundary; the other branches are conservative
overcounts.  Each resulting component with at least three vertices is a
simple cycle.  If a split isolates one virtual subarc together with a newly
created real edge, the result is a compressed two-vertex component with
exactly one virtual side and one real side.  The search retains this component
and continues across its real side.  Thus the virtual subarcs retain the order
of all root-arc vertices encountered by the search, and every genuine
unfilled sector remains represented.

For the counting lemmas, a layer-$k$ triangle is counted only after all
triangles relevant to the nearest-face test are already closed in the listed
state.  The reason for the two closure thresholds is the following
shortest-path observation.  If a competing $4^+$-face is closer than the
root face, then the last triangular face before it on a shortest dual path is
at distance at most $k-2$ from the counted triangle; if a competing
$4^+$-face is tied with the root face, the corresponding last triangular
face is at distance at most $k-1$.  Thus the nearest-face count closes
triangles through distance $k-2$, while the unique-nearest count closes
triangles through distance $k-1$.  Lemma~\ref{lem:radius-five} gives the
formal witness statement used in all nearest and unique-nearest counts.

\subsection{Verification of certificate \CertOne: closed triangular neighbourhoods}

The pseudocode blocks in the remainder of this appendix correspond, in order, to the finite
certificates \CertOne--\CertSix{} in Theorem~\ref{thm:finite-certificates}.  For
navigation, the verification record begins with a theorem-to-code dictionary mapping
each certificate to the source file, command line, output fields, and
certificate files used to certify it.  The sealed certificate checker and the
Boost planarity cross-check audit existing computations rather than define
additional mathematical searches, so they do not have separate pseudocode
blocks.

\begin{algorithm}[t]
\caption{\textsc{ClosedRadiusThreeSearch}}
\label{alg:closed-radius-three}
\begin{algorithmic}[1]
\State Start from the root triangle $(0,1,2)$ in layer $0$.
\While{there is an edge incident with exactly one listed triangle of layer at most $2$}
  \State Add a triangle across that edge.
  \State Try one new third vertex and each old third vertex distinct from the edge endpoints.
  \State Reject states with repeated triangles, an edge with three listed incidences, an $8$-cycle, or $E>3V-6$.
\EndWhile
\State Record every completed layer-three state, whether a layer-three edge remains open, and the number of vertices in the listed patch.
\end{algorithmic}
\end{algorithm}

\begin{lemma}[Closed radius-three triangular neighbourhoods]
\label{lem:closed-radius-three}
Let $G$ be a simple 2-connected $C_8$-free plane graph with
$|V(G)|\ge8$, and let $g$ be a 3-face of $G$.  Define the
\defterm{radius-three dual ball} around $g$ by
\[
  \Ball_3(g)=\{h:\ d^*(g,h)\le 3\}.
\]
If every face in $\Ball_3(g)$ is triangular, then no face lies at dual distance $4$ from $g$.
Consequently every face of $G$ lies in $\Ball_3(g)$, and $G$ has at
most $7$ vertices.
\end{lemma}

\begin{proof}
The search \textsc{ClosedRadiusThreeSearch} uses only necessary rejection
rules, so by Proposition~\ref{prop:search-state-domination} every genuine
all-triangular radius-three patch occurs among the enumerated states.  Notice
that no minimum-degree hypothesis is used here: the search enumerates a
slightly larger class than the one needed later, allowing boundary and
degree-two phenomena that would be excluded in the final reduced graph.

The finite verification for the $C_8$, radius-three case certifies
\[
\begin{array}{c|c}
\text{certified quantity} & \text{value}\\
\hline
\text{enumerated states} & 291\\
\text{completed radius-three states} & 85\\
\text{completed states with an open layer-three edge} & 0\\
\text{maximum number of vertices in the completed list} & 7 .
\end{array}
\]
Thus every completed state reaching layer three has no open edge incident
with a layer-three triangle.
Thus a genuine all-triangular radius-three neighbourhood cannot have an
adjacent face at distance four.  Indeed, such a face would have to lie across
an open edge of a listed layer-three triangle; if that edge were already
closed, the face on the other side would already be part of the displayed
radius-three patch.
Since the dual graph of a connected plane graph is connected, the absence of a
face at dual distance $4$ means that every face of $G$ already lies in
$\Ball_3(g)$.
The same complete $C_8$ radius-three list has at most seven vertices.
Hence \textsc{ClosedRadiusThreeSearch} verifies certificate \CertOne.
\end{proof}

This is the finite verification underlying certificate \CertOne.  The main text
derives Lemma~\ref{lem:covering} from \CertOne{} and
Proposition~\ref{prop:lifting-domination}.

\subsection{Verification of certificate \CertTwo: small root-face bounds}

\begin{lemma}\label{lem:same-boundary}
Let $P$ be a connected plane patch whose boundary components record the
unfilled complementary regions.  Suppose a triangular face $uvw$ is attached
across a boundary edge $uv$ of $P$.  If $w$ is already a vertex of
$P$, then $w$ lies on the same boundary component as the edge $uv$.
\end{lemma}

\begin{proof}
The new face lies in the unfilled complementary region incident with the
boundary edge $uv$.  If the third vertex $w$ were an old vertex on a
different boundary component, then the embedded triangle $uvw$ would have to
cross the existing patch or leave this complementary region.  Both are
impossible in a plane extension of $P$.  Thus $w$ is on the same boundary
component.
\end{proof}

\begin{lemma}\label{lem:shared-face-path}
In a simple 2-connected plane graph with minimum degree at least $3$, two
distinct faces cannot share a boundary path of length at least two.
\end{lemma}

\begin{proof}
If two faces shared $v_0v_1\cdots v_k$ with $k\ge 2$, then the two sides
between the consecutive edges $v_0v_1$ and $v_1v_2$ would already be
occupied by the two face interiors.  In the cyclic order of edges around the
internal vertex $v_1$, the edges $v_1v_0$ and $v_1v_2$ would therefore
already bound both incident face angles, leaving no room for any third edge at
$v_1$.  Hence $d(v_1)=2$, contradicting $\delta(G)\ge 3$.
\end{proof}

\begin{algorithm}[t]
\caption{\textsc{SmallRootFaceSearch}($d$)}
\label{alg:small-root}
\begin{algorithmic}[1]
\State Start from the root $d$-face $(0,1,\ldots,d-1)$ in layer $0$.
\State Maintain the boundary components of the listed patch.
\While{a real boundary edge $uv$ may receive a triangle}
  \State Try one new third vertex and each old third vertex on the same boundary component.
  \State Assign the new layer from a minimum-layer shared neighbour.
  \State Reject states with repeated triangles, an edge with three listed incidences, a new triangle sharing two edges with a listed face, a layer above $5$, a non-minimum-layer attachment, or an $8$-cycle.
  \State Count a layer-$k$ triangle, $k\le 3$, only when every listed triangle within dual distance $k-2$ from it is closed.
\EndWhile
\State Ignore completed states with fewer than eight vertices.
\State Record, for each level $t$, the number of surviving states and the maximum count.
\end{algorithmic}
\end{algorithm}

\begin{proof}[Verification of certificate \CertTwo]
By Proposition~\ref{prop:search-state-domination},
the search in Algorithm~\ref{alg:small-root} includes every genuine rooted radius-five
patch around a $d$-face.  It may also keep extra states, which can only
increase the certified maximum.  Lemma~\ref{lem:radius-five}
with $k\le3$ shows that this radius and the $k-2$ closure condition are
sufficient for the nearest-face test: any closer competing $4^+$-face would
leave an open triangular witness at distance at most $k-2$ from the counted
triangle, and all such witnesses lie within dual distance at most $5$ from
the root.

The implementation assigns count $0$ to a completely closed state with
fewer than eight vertices.  Such a state cannot be the rooted patch of a graph
satisfying $|V(G)|\ge 8$: once no boundary edge remains, the listed patch
together with the root face is already the whole connected plane graph.

The complete level-by-level tables are part of the verification record.
For the proof, only the certified maxima and terminal exhaustion levels are
needed:
\[
\begin{array}{c|cccc}
d&4&5&6&7\\
\hline
\max b(f)\text{ found by the search}&7&7&5&4\\
\text{first terminal level }t\text{ with no surviving states}&11&13&12&11.
\end{array}
\]
Since every genuine $f$-bad triangle is counted in some enumerated state,
these exhaustive-search maxima certify \CertTwo.
\end{proof}

\subsection{Verification of certificate \CertThree: unique-nearest four-face count}

For certificate \CertThree, let
\[
  u(f)=|\{g:\ \Near(g)=\{f\}\}|
\]
be the number of triangular faces whose unique-nearest $4^+$-face is $f$.

\begin{lemma}[Unique-nearest $4$-face bound]\label{lem:unique-four}
Let $G$ be a simple 2-connected $C_8$-free plane graph with
$\delta(G)\ge 3$ and $|V(G)|\ge 8$, and let $f$ be a 4-face.  Then
$u(f)\le 6$.
\end{lemma}

\begin{algorithm}[t]
\caption{\textsc{UniqueNearestFourFaceSearch}}
\label{alg:unique-four}
\begin{algorithmic}[1]
\State Run \textsc{SmallRootFaceSearch}$(4)$ with allowed layers extended from $5$ to $6$.
\State Count a layer-$k$ triangle, $k\le 3$, only when every listed triangle within dual distance $k-1$ from it is closed.
\State Ignore completed states with fewer than eight vertices.
\State Record, for each level $t$, the number of surviving states, the best nearest count, and the best unique-nearest count.
\end{algorithmic}
\end{algorithm}

\begin{proof}
The search in Algorithm~\ref{alg:unique-four} is the radius-six version of
Algorithm~\ref{alg:small-root} with root degree $4$.  The extra layer is
needed only for the
unique-nearest test.  By Lemma~\ref{lem:radius-five}, any competing
$4^+$-face at distance at most $k=d^*(f,g)\le3$ from a counted triangle
$g$ would leave an open triangular witness at distance at most $k-1$ from
$g$.  Such witnesses lie within dual distance at most $5$ from $f$, and
the triangular faces needed to close their open edges have layer at most
$6$.  Hence every genuine triangle with $\Near(g)=\{f\}$ satisfies the
closure condition used by the search and is therefore counted in some
enumerated state.  As in
Lemma~\ref{lem:small-degree}, completed states with fewer than eight vertices
are ignored because they cannot occur inside a connected graph with at least
eight vertices.

The finite verification gives the following certificate:
\[
\begin{array}{c|c}
\text{certified quantity} & \text{value}\\
\hline
\text{counted triangles with } f\in\Near(g) & 7\\
\text{counted triangles with } \Near(g)=\{f\} & 6\\
\text{first terminal level} & 11 .
\end{array}
\]
The full level-by-level table is part of the verification record.
Hence the maximum unique-nearest count is $6$, verifying certificate \CertThree.
\end{proof}

\subsection{Verification of certificate \CertFour: the dangerous four-face case}

\begin{lemma}[Dangerous $4$-face certificate]\label{lem:dangerous-four}
Let $G$ be a simple 2-connected $C_8$-free plane graph with
$\delta(G)\ge3$ and $|V(G)|\ge8$, and let $f$ be a $4$-face.  Suppose
that exactly six triangular faces $g$ satisfy $\Near(g)=\{f\}$, and suppose
that there is a further triangular face $h$ with $f\in \Near(h)$.  Then
\[
  |\Near(h)|\ge 3.
\]
\end{lemma}

\begin{algorithm}[t]
\caption{\textsc{RefinedFourFaceLoadSearch}}
\label{alg:refined-four}
\begin{algorithmic}[1]
\State Run the radius-six rooted $d=4$ search used for the unique-nearest
test.
\State Mark a state as dangerous if its nearest count is $7$, its
unique-nearest count is $6$, and it has exactly one non-unique counted
triangle.
\For{each dangerous state}
  \State Find the unique non-unique triangle $h$.
  \State Inspect the open triangular witnesses at dual distance at most
  $d^*(f,h)-1$ from $h$.
  \State Check that there is exactly one such witness triangle $t$, that
  $t$ has exactly two open edges, and that neither open edge admits a legal
  triangular continuation.
\EndFor
\State Output the number of dangerous states and the number satisfying the
certificate above.
\end{algorithmic}
\end{algorithm}

\begin{proof}
By Lemma~\ref{lem:small-degree}, at most seven triangles contribute to
$\ell(f)$.  Hence the six unique-nearest triangles and the further triangle
$h$ in the hypothesis are exactly the seven counted triangles, and $h$ is
the unique non-unique contributor.  The statement is local in the same rooted
$d=4$ state space used in Lemma~\ref{lem:unique-four}.  The verification uses
the same radius-six
state space as the unique-nearest computation, so the dangerous states are
tested after all witnesses needed for the unique-nearest test are visible.
The finite verification
\textsc{RefinedFourFaceLoadSearch}, documented in the verification record,
certifies the following facts:
\[
\begin{array}{c|c}
\text{quantity} & \text{certified value}\\
\hline
\text{dangerous states with }(b,u)=(7,6) & 112\\
\text{dangerous states satisfying the two-open-edge certificate} & 112\\
\text{dangerous states without certificate} & 0 .
\end{array}
\]
Here $b$ denotes the number of counted triangles $g$ with $f\in \Near(g)$,
and $u$ denotes the number of counted triangles with $\Near(g)=\{f\}$.
The verification record also gives a representative certified
dangerous state, including the non-unique triangle, the witness triangle, and
the two open witness edges.

Thus every possible rooted dangerous configuration has the following
structure.  There is a unique counted triangle $h$ for which
$f\in \Near(h)$ but $\Near(h)\ne\{f\}$.  This triangle has layer $2$.  There is
exactly one open triangular witness $t$ at dual distance $1$ from $h$,
the witness $t$ has exactly two open edges, and neither of these open edges
can be completed by attaching another triangular face without violating the
necessary local conditions, in particular $C_8$-freeness.  In this last
one-step test the program does not impose the non-minimum-layer attachment
rejection from the rooted enumeration; it asks whether any triangular face
can lie across the open edge in a plane $C_8$-free realization.

Now consider such a configuration inside the ambient graph $G$.  The two
witness edges are non-root edges of $t$: root edges are excluded from the
witness list, since the face across a root edge is the root face $f$ itself
and cannot supply an additional competing $4^+$-face.  Since an open
witness edge cannot be incident with another triangular face in $G$, the
face on the other side of that edge is not triangular.  The graph is
2-connected, so every face boundary is a cycle, and the graph is simple, so
there are no faces of degree $1$ or $2$.  Hence the face on the other
side of each open witness edge is a $4^+$-face.  Each of these two
$4^+$-faces is reached from $h$ by first moving from $h$ to $t$, and
then crossing one open edge of $t$.  Hence each lies at dual distance at
most $2=d^*(f,h)$ from $h$.  Since $f\in \Near(h)$, no $4^+$-face is
closer to $h$ than $f$, so both of these faces are also nearest
$4^+$-faces to $h$.

The two faces obtained from the two open edges are distinct.  If they were
the same face, then that face and the triangular face $t$ would share a
boundary path of length two, contradicting Lemma~\ref{lem:shared-face-path}.  They
are also distinct from $f$, because $t$ has positive layer and its open
edges are not root edges.  Therefore $h$ has at least the three nearest
$4^+$-faces consisting of $f$ and these two competitors.  Hence
$|\Near(h)|\ge3$.
This is also the lifting step from a certified state to every ambient
completion: the certificate supplies two specific real open edges of the
listed witness $t$, and plane incidence supplies their true second facial
sides.  The one-step continuation test excludes a triangle on either side;
simplicity and 2-connectivity make both sides $4^+$-faces, while
Lemma~\ref{lem:shared-face-path} keeps their identities distinct.  No identity
of an unlisted face is inferred merely from a seal or from the representative
drawing.
This verifies certificate \CertFour.
\end{proof}

\subsection{Verification of certificate \CertFive: the large-face single-edge bound}

We now control the equal-split load on larger faces.  We use the notation
$S_f(g)$ and $L_f(e)$ introduced in Section~\ref{sec:local-inputs}: the
first records the shortest entry edges from the root face $f$ to $g$, and
the second is the portion of $\ell(f)$ carried by one boundary edge $e$.
This subsection describes the finite sealed search that certifies the
one-edge bound.

\begin{lemma}[Radius-five witness]\label{lem:radius-five}
Let $G$ be a simple 2-connected $C_8$-free plane graph with
$\delta(G)\ge 3$ and $|V(G)|\ge 8$.  Let $f$ be a $4^+$-face, let
$g$ be a triangular face, and put $k=d^*(f,g)$.  Assume $k\le3$.  Let
$P$ be a rooted triangular subpatch of $G$ whose root is $f$, which
contains $g$, and whose displayed incidences agree with those of $G$.
Then:
\begin{enumerate}
\item if every triangular face of $P$ at dual distance at most $k-2$ from
$g$ is closed in $P$, then no $4^+$-face of $G$ is closer to $g$
than $f$;
\item if every triangular face of $P$ at dual distance at most $k-1$ from
$g$ is closed in $P$, then $f$ is the unique-nearest $4^+$-face to
$g$.
\end{enumerate}
All triangular witnesses needed for these tests lie within dual distance at
most $5$ from $f$.
\end{lemma}

\begin{proof}
Suppose that a non-root $4^+$-face is tied with or closer to $g$ than
$f$.  Follow a shortest path from $g$ to such a face and truncate it at
the first $4^+$-face encountered.  This first $4^+$-face cannot be $f$:
reaching $f$ already takes $k$ steps, so any later non-root competitor
would be farther than $f$.  Write the truncated path as
\[
  g=g_0,g_1,\ldots,g_{m-1},h,
\]
where $h\ne f$ is a competing $4^+$-face.  By the choice of the first
$4^+$-face, all $g_i$ are triangular.  We first justify that the relevant closure
condition forces them to be listed, rather than assuming this.  The face
$g_0=g$ is listed.  Inductively, if $g_i$ is listed and lies within the
closure radius, the edge shared with the actual triangular face $g_{i+1}$
cannot be closed while $g_{i+1}$ is absent: faithful incidence means that
closing this edge lists its true second incident face.  Hence closure forces
$g_{i+1}$ to be listed.  This proves by induction that
$g_0,\ldots,g_{m-1}$ are listed whenever the stated closure radius reaches
them.

The edge shared by $g_{m-1}$ and the unlisted $4^+$-face $h$ is then
open, contradicting closure of $g_{m-1}$.  If $h$ is strictly closer than
$f$, then $m\le k-1$, so $m-1\le k-2$, proving (1).  If $h$ is tied
with or closer than $f$, then $m\le k$, so $m-1\le k-1$, proving (2).
The witnesses inspected in (1) or (2) are at distance at most
$k+(k-1)\le5$ from $f$.
\end{proof}

The verifier for $L_f(e)$ is rooted at a three-edge window of the root
face.  After relabelling, the distinguished boundary edge is $01$, and the
adjacent root edges $30$ and $12$ are also kept explicitly.  The remaining
arc of the root boundary is initially represented by the virtual side $23$.
A triangular continuation whose new third vertex lies on this omitted arc
inserts that vertex as an anchor and replaces the containing virtual side by
two ordered virtual subarcs.  Virtual sides are bookkeeping sides, not graph
edges, and are ignored by the $C_8$-test unless a later triangular
continuation creates a real edge between their endpoints.

The two non-distinguished root edges $30$ and $12$ carry $O$-seals.  An
$O$-seal is only a nonblocking marker for a displayed adjacent root entry: it
is not a non-root face, and triangular continuations may still touch that
edge.  Every relevant open real boundary edge is handled in exactly one of
the following ways:
\begin{enumerate}
\item the second incident face is triangular, and the search adds that
triangle using an admissible old vertex, an ordinary new vertex, or a new
virtual anchor on one of the current virtual subarcs;
\item the second incident face is a non-root $4^+$-face, recorded by an
$H$-seal;
\item the edge is another entry edge of the same root face, recorded by an
$R$-seal.
\end{enumerate}
The $R$-branch is not allowed on the displayed root-window edges
$30,01,12$, since the root face cannot lie on both sides of one of its own
boundary edges.  For a completed sealed state $S$, the verifier evaluates
the total $\sum_g\lambda_S(g)$, where a selected listed triangle has
\[
  \lambda_S(g)=\frac{p_S(g)}{q_S(g)s_S(g)}
\]
and every unselected triangle has $\lambda_S(g)=0$.
The quantities $p_S(g)$, $s_S(g)$, and $q_S(g)$ are the visible
distinguished-entry, visible shortest-entry, and merged-nearest-face counts
described in Section~\ref{sec:local-inputs}.  The finite certificate proves directly that
every terminal state in this three-root-edge sealed search has total value at
most $9/2$.  Thus these terminal values themselves give the bound.

\begin{algorithm}[t]
\caption{\textsc{SealedThreeRootEdgeSearch}}
\label{alg:sealed-single-edge}
\begin{algorithmic}[1]
\State Start with boundary $(3,0,1,2)$, distinguished root edge $01$,
displayed adjacent root entries $30$ and $12$, and virtual side $23$.
\State Mark $30$ and $12$ with nonblocking $O$-seals.
\While{there is an unresolved relevant open edge}
  \State Form the list of triangles $g$ whose sealed root distance is
  $k\le3$, whose nearest-face test is not already beaten by an $H$-seal,
  and which have a displayed shortest entry through $01$.
  \State If this list is empty, choose the lexicographically least unresolved
  displayed root-window edge; otherwise form the set of open edges incident
  with a listed triangle at listed-patch distance at most $k-1$ from one of
  these $g$, and choose the lexicographically least edge in that set.
  \State Across the chosen edge, branch into all legal old-vertex and ordinary
  new-vertex continuations, all new virtual-anchor continuations obtained by
  inserting the third vertex into a current virtual subarc of the selected
  edge's boundary component, and the
  appropriate $H$- and $R$-seals.
  \State Forbid triangular continuations through $H$- or $R$-sealed
  edges, but allow continuations touching an $O$-sealed edge.
  \State Reject non-plane boundary order, repeated triangles, an edge with
  three facial incidences, a new triangle sharing two edges with a listed
  face, a construction layer above $6$, or an $8$-cycle.
\EndWhile
\State In each completed state, identify all $H$-seals as one non-root
$4^+$-face and keep the $R$-seals as distinct root-entry edges.
\State Merge the visible $H$-competitors as specified above and compute the
exact rational upper-bound value.
\State Output the maximum terminal value and the number of terminal states
with value above $9/2$.
\end{algorithmic}
\end{algorithm}

\begin{lemma}[Compression and seal alternatives]
\label{lem:compression-seal-alternatives}
In the three-root-edge model for $L_f(e)$, the compressed root boundary and
the $H/R/O$-seal branches include every true local configuration that can
contribute through $e$, and may also include extra states that only increase
the computed upper bound.  More precisely, every relevant open real
boundary edge has exactly one true second side: a triangular face, a non-root
$4^+$-face, or another entry of the same root face $f$.  These are
represented respectively by a triangular continuation, an $H$-seal, and an
$R$-seal.  The triangular alternative includes the old-vertex, ordinary
new-vertex, and virtual-anchor branches of
Lemma~\ref{lem:faithful-cut-open}.  The $O$-seals on the two adjacent
displayed root edges are only nonblocking root-entry markers.
\end{lemma}

\begin{proof}
In a $2$-connected plane graph every edge has two facial sides and face
boundaries are cycles.  Once one side of an open real boundary edge is already
listed, the other side is either a triangle, a non-root face of degree at
least four, or the original root face $f$.  These alternatives are mutually
exclusive in the ambient embedding and are exactly the branches of the sealed
search.  For the triangular alternative, the virtual-anchor branches try every
current virtual subarc; exactly one branch records the correct position when
the third vertex lies on the omitted root arc.  The compressed sides of the
root face are bookkeeping sides, not graph edges, and an $O$-seal does not
block a triangular continuation.  The remaining branches only enlarge the
search space.
\end{proof}

\begin{lemma}[Unresolved-edge exposure]
\label{lem:unresolved-edge-exposure}
Fix a true local configuration affecting $L_f(e)$.  If the sealed search
chooses one unresolved relevant edge before another, then the true
configuration still has an exposure order following that choice until all
relevant edges have been resolved.
\end{lemma}

\begin{proof}
Let $K\rightsquigarrow S$ denote the lifting relation of
Lemma~\ref{lem:faithful-cut-open}.  Strengthen it by marking every genuine
unresolved sector side whose second face can lie on a shortest root-to-candidate
path or can witness a closer or tied $4^+$-face.  Its image in $S$ is a real
open edge on the boundary of the corresponding search sector.  This marked
edge invariant is initially true.

Resolve the edge selected by the deterministic rule.  Its true second side is
one of the mutually exclusive alternatives in
Lemma~\ref{lem:compression-seal-alternatives}.  A triangular resolution splits
one disk sector into disk sectors; an old-vertex identification or a
virtual-anchor insertion merely records the already fixed endpoint order on
that sector boundary.  An $H$- or $R$-resolution closes only the selected
side.  In every case, any other marked genuine side remains a real side of
exactly one resulting sector, with the same true second face.  Boundary
splitting cannot turn it into a virtual side, duplicate it, or identify its
endpoints, because virtual sides contain no real edge and the ambient graph is
simple.  Its displayed relevance may disappear only if the newly exposed
object supplies a strictly shorter competitor or proves that the contribution
through $e$ is zero; then the side is no longer required for an upper bound.
Otherwise the same shortest-path or witness certificate remains, so the edge
remains a marked eligible edge for the deterministic relevance rule.

For the selected edge, the candidate rule enumerates every possible true
third vertex: a new interior vertex, an already listed occurrence on the same
sector boundary, or the unique virtual subarc containing an unlisted root
vertex.  These are respectively the new-vertex, old-vertex, and virtual-anchor
branches.  Thus one child $S'$ satisfies $K\rightsquigarrow S'$ and the
marked-edge invariant.  Induction on the number of unresolved marked sides
gives an exposure in the deterministic order.  This also proves the required
forward completeness of the dynamic relevance and candidate rules; it does
not rely on the layer audit.
\end{proof}

\begin{lemma}[Forward simulation through the layer cutoff]
\label{lem:forward-simulation-cutoff}
Let $K$ be a partial exposure of a genuine configuration affecting
$L_f(e)$, and let $S$ represent $K$ under the lifting invariant of
Lemma~\ref{lem:faithful-cut-open}.  If $S$ is not terminal and the
deterministic rule selects a relevant edge $a$, then there is a child $S'$
and the true resolution $K'$ of $a$ such that $S'$ represents $K'$.
Iterating gives either a terminal representative or a state at which the true
child would have construction layer above six.  The latter alternative does
not occur in the certified transition tree.
\end{lemma}

\begin{proof}
The true second side of $a$ is unique.  If it is nontriangular, it is either
the root face or a non-root $4^+$-face and hence is represented by the
corresponding $R$- or $H$-transition.  If it is triangular, its third
vertex is respectively new, already listed on the same disk-sector boundary,
or unlisted in the interior of exactly one virtual subarc.  These are exactly
the new-vertex, old-vertex, and virtual-anchor transitions.  The lifting
invariant proves that the chosen child preserves injectivity, cyclic order,
and all genuine incidences.  The rejection rules cannot reject this child:
repeated faces, a third facial incidence, a non-plane split, and a real
$8$-cycle contradict the ambient hypotheses; the shared-path rejection
contradicts $\delta(G)\ge3$; and the construction-layer update is exactly
the one used by the selected true resolution.  Lemma~\ref{lem:unresolved-edge-exposure} shows
that resolving $a$ first cannot destroy another required route or witness.
Key bisimulation in Proposition~\ref{prop:search-state-domination} shows that
deduplication retains a child with the same future transitions.

This proves the first assertion by induction on unresolved relevant sides.
The only remaining rejection is construction layer above six.  The
no-layer-limit retry recorded in
Lemma~\ref{lem:finite-guarded-sealed-certificate} is performed at every
reachable representative and tries every admissible third vertex at the
selected edge; it reports no otherwise legal relevant child above layer six.
This last sentence is exactly the finite assertion used here.  It does not
claim that the audit alone proves the lifting invariant: the transfer from a
genuine resolution to an enumerated child is the preceding mathematical
simulation.
\end{proof}

\begin{lemma}[Merged competitors preserve all three load factors]
\label{lem:merged-h-denominator}
Let $G$ satisfy the hypotheses of Lemma~\ref{lem:covering}, let
$f$ be a $4^+$-face of $G$, and let $S$ be a terminal sealed state
representing a genuine completion around $f$.  Let $e$ be the displayed
root edge.  For every listed triangle $g$,
\[
 c_{f,e}(g)
 \le \lambda_S(g).
\]
Here $\lambda_S(g)=p_S(g)/(q_S(g)s_S(g))$ for a triangle selected by
the verifier and $\lambda_S(g)=0$ otherwise.
In particular, identifying all $H$-sealed faces as one node is conservative
not only for $q_S(g)$, but also for the breadth-first computations defining
$p_S(g)$ and $s_S(g)$.
\end{lemma}

\begin{proof}
Let $D_S$ be the sealed dual graph after the $H$-nodes have been
identified, let $F$ and $H$ denote its root and merged competitor nodes,
and set
\[
 r=d_{D_S}(F,g),\qquad h=d_{D_S}(H,g).
\]
Every shortest path may be chosen simple.  If a shortest path from $F$ to
$g$ uses $H$, its suffix from $H$ to $g$ has length less than $r$;
thus $h<r$.  The same argument applies to a shortest path from a visible
entry neighbour $a$ satisfying $1+d_{D_S}(a,g)=r$: if that path uses
$H$, then again $h<r$.

Suppose first that the true contribution through $e$ is nonzero.  A shortest
dual path leaving $f$ through $e$ is contained in the represented local
configuration and gives a path in $D_S$ that avoids $H$.  If $h<r$,
choose a simple shortest $H$-to-$g$ path.  It leaves $H$ through one
actual $H$-seal and never returns to $H$, so its remaining edges give a
genuine dual path from the corresponding non-root $4^+$-face to $g$ of
length $h$.  This face would be strictly closer than $f$, a contradiction.
Hence $h\ge r$.

It follows from the first paragraph that the shortest paths used in the
definitions of $r$, $p_S(g)$, and $s_S(g)$ avoid $H$.  They are genuine
visible dual paths.  The true path through $e$ and a shortest visible root
path show, in opposite directions, that the true root distance equals $r$.
A nonzero contribution has $f\in\Near(g)$, so
Lemma~\ref{lem:covering} gives $r=d^*(f,g)\le3$.  Together with
$h\ge r$, this shows that $g$ passes the distance and competitor tests.
The displayed entry is counted by $p_S(g)$, and every entry counted by
$s_S(g)$ is a true shortest entry.  In particular $p_S(g)>0$, so $g$
is selected and $\lambda_S(g)=p_S(g)/(q_S(g)s_S(g))$.  Therefore
\[
 \mathbf 1_{\{e\in S_f(g)\}}\le p_S(g),
 \qquad |S_f(g)|\ge s_S(g).
\]
If $h=r$, the simple $H$-to-$g$ path constructed above shows that one
actual non-root $4^+$-face is tied with $f$, so
$|\Near(g)|\ge2=q_S(g)$.  If $h>r$, then $q_S(g)=1$ and
$|\Near(g)|\ge q_S(g)$ is immediate.  Multiplying the three inequalities
gives the claim for a nonzero true contribution.

If the true contribution is zero, the left-hand side is zero.  This also
covers every triangle omitted from the sealed sum, including the case
$h<r$.  Hence the inequality holds in all cases.
\end{proof}

\begin{proposition}[Sealed states upper-bound genuine completions]
\label{prop:sealed-completion-domination}
Fix a boundary edge $e$ of a face $f$ with $d(f)\ge9$.  For every true
radius-six local configuration that can affect $L_f(e)$, the three-root-edge
sealed search has a terminal state $S$ whose computed value
\[
  \sum_g \lambda_S(g)
\]
is at least the true contribution of all triangles in that configuration to
$L_f(e)$.
\end{proposition}

\begin{proof}
Compress the boundary of $f$ so that the distinguished edge $e$ is the
displayed edge $01$, the two neighbouring root edges are $30$ and $12$,
and the rest of the boundary of $f$ is initially the virtual side $23$.
By Lemma~\ref{lem:faithful-cut-open} and
Proposition~\ref{prop:search-state-domination}, every true triangular
continuation across a relevant real open edge is generated by an old-vertex,
ordinary new-vertex, or virtual-anchor branch of the search.  The anchor branch
records the order of every encountered vertex on the omitted root arc.  In
particular, a sector cut off by a real edge and a virtual root subarc remains
represented by its compressed two-vertex component, with no duplicated
displayed vertex and no artificial graph edge.
Lemma~\ref{lem:finite-guarded-sealed-certificate} verifies that no relevant
triangular continuation is lost to the construction-layer bound.  By
Lemma~\ref{lem:compression-seal-alternatives}, the
non-triangular possibilities are exactly the $H$- and $R$-seal branches,
and an $O$-seal never blocks a true triangular continuation.  By
Lemma~\ref{lem:unresolved-edge-exposure}, the deterministic choice of which
unresolved relevant edge to resolve next does not remove a true completion.

Lemma~\ref{lem:merged-h-denominator} now compares the terminal value with the
true fractional load one triangle at a time, including the possible effect of
the merged $H$-node on all breadth-first distances:
\[
 c_{f,e}(g)
 \le \lambda_S(g).
\]
The search terminates only after every open edge that can affect these
denominators has been closed by a listed triangle or resolved by an $H$- or
$R$-seal;
by Lemma~\ref{lem:radius-five}, these are exactly the witnesses relevant to
triangles with $d^*(f,g)\le3$.  Summing the displayed inequality over the
listed triangles gives the claimed upper bound.
\end{proof}

\begin{lemma}[Finite sealed certificate]
\label{lem:finite-guarded-sealed-certificate}
The finite search tree specified in Algorithm~\ref{alg:sealed-single-edge}
has a certified terminal bound with the following properties.
\begin{enumerate}
\item The recorded search has exactly $223766$ labelled states, $223765$
parent-child edges, one root state, and $136987$ terminal leaves.  Every
non-root state has a recorded parent, every recorded child is obtained by one
legal transition of the sealed search, including every virtual-anchor
transition, and the terminal leaves are exactly the states with no legal child.
\item A stored construction layer may exceed the current sealed root distance,
but never underestimates it.  The certificate records $201940$ states with
such a difference, $554355$ differing face records, and maximum excess $5$.
Within the transition system specified above, at every reachable state the verifier selects the same unresolved relevant
edge used by the transition relation and retries every admissible third-vertex
choice there with the construction-layer bound relaxed.  It finds no legal
relevant triangular branch above layer $6$.  Together with the forward
completeness supplied by Lemma~\ref{lem:unresolved-edge-exposure}, this proves
that the layer bound omits no genuine relevant branch.
\item Every virtual-anchor transition records the new label between the
endpoints of one current virtual subarc before splitting the boundary.  A
compressed two-vertex boundary consisting of one real side and one virtual
subarc is retained as a search component.  Such a component occurs in
$88838$ states, with $124347$ such components in total.  The audit finds no
invalid boundary component.
\item Every terminal state has computed upper-bound load at most $9/2$.  There
are no overbound terminal states.
\end{enumerate}
The mathematical certificate assertions and accompanying regression totals are
\[
\begin{array}{c|c}
\text{quantity} & \text{certified value}\\
\hline
\text{regression totals: states / edges / leaves}
  &223766/223765/136987\\
\text{construction-layer underestimates}
  &0\\
\text{omitted relevant layer branches}
  &0\\
\text{invalid boundary components}
  &0\\
\text{computed-load maximum / overbound terminals}
  &9/2\;/\;0 .
\end{array}
\]
Only exhaustion, transition validity, the zero audit counts, and the load
bound are mathematical assertions.  The exact state, edge, and leaf totals
are consistency checks for this finite computation.
\end{lemma}

\begin{proof}
This lemma is a finite statement about the explicitly specified transition
system.  Its verification does not assume
Proposition~\ref{prop:sealed-completion-domination}.  The generator records the
reachable transition tree and the theorem-level audit fields, and the archived
certificate is then checked as follows.  Proposition
\ref{prop:sealed-completion-domination} is used only to interpret the certified
terminal maximum as an upper bound for genuine plane graphs.  The
theorem-to-code dictionary records the source programs, commands,
outputs, and certificate files.

The C++ certificate-checking program is a re-parser and transition re-checker.
It reads every recorded state, recomputes its canonical labelled key, verifies
the rooted tree and all
parent-child transitions, and recomputes every terminal load.  It also
recomputes the construction-layer, sealed-distance, no-layer-limit branch, and
compressed-boundary audits from every representative state.  It reports zero
construction-layer underestimates, zero omitted relevant branches, zero
invalid boundary components, zero child or terminal-load mismatches, and no
overbound terminal state.

The source package additionally contains a C++ malformed-certificate test
driver.  It creates temporary corrupted copies of the archived certificate and
checks that the certificate checker rejects truncated node and edge tables,
duplicate or negative identifiers, out-of-range children, extra TSV fields,
malformed overbound rows, and embedded NUL bytes.  These defensive tests
exercise the certificate grammar and failure paths; the semantic transition
and load checks remain the responsibility of the C++ certificate checker.
Fixed state counts and fingerprints remain consistency checks rather than
mathematical assumptions.
\end{proof}

\begin{lemma}[Sealed single-edge verification]
\label{lem:single-edge-capacity}
Let $G$ be a simple $2$-connected $C_8$-free plane graph with
$\delta(G)\ge3$ and $|V(G)|\ge8$.  Let $f$ be a face of degree at
least $9$, and let $e\in E(\partial f)$.  Then
\[
  L_f(e)\le \frac92.
\]
\end{lemma}

\begin{proof}
By Lemma~\ref{lem:covering}, every triangle contributing to $L_f(e)$ has
dual distance $k\le3$ from $f$.  For such a layer-$k$ triangle $g$,
Lemma~\ref{lem:radius-five} says that the witnesses relevant to closer or
tied $4^+$-faces are incident with triangular faces at dual distance at most
$k-1$ from $g$.  Since $k\le3$, any triangular continuation needed to
expose such a witness has layer at most $k+(k-1)+1\le6$.  Hence the sealed
search is finite.  By Proposition~\ref{prop:sealed-completion-domination}, the
maximum terminal value produced by the sealed search is at least the true
value of $L_f(e)$ for every ambient $C_8$-free graph.  The finite
certificate in Lemma~\ref{lem:finite-guarded-sealed-certificate} gives maximum
terminal value $9/2$, so $L_f(e)\le9/2$.
This verifies certificate \CertFive.
\end{proof}

\subsection{Verification of certificate \CertSix: the eight-vertex base case}

\begin{algorithm}[t]
\caption{\textsc{BaseCaseN8Search}}
\label{alg:base-n8}
\begin{algorithmic}[1]
\For{$m\in\{17,18\}$}
  \State Enumerate all labelled simple graphs on vertex set $\{1,\ldots,8\}$ with exactly $m$ edges.
  \For{each enumerated graph $G$}
    \State Search for a Hamiltonian cycle in $G$.
    \If{a Hamiltonian cycle is found}
      \State Discard $G$, since on eight vertices this is exactly a copy of $C_8$.
    \Else
      \State Apply the induced-subgraph sparsity test $e(H)\le 3|V(H)|-6$ for every induced $H$ with at least three vertices.
      \If{the sparsity test does not reject $G$}
        \State Run the exhaustive memoized deletion/contraction search for a $K_5$ or $K_{3,3}$ minor.
        \State Record whether $G$ remains planar and non-Hamiltonian.
      \EndIf
    \EndIf
  \EndFor
\EndFor
\State Certify that no planar non-Hamiltonian graph remains for $m=17$ or $m=18$.
\end{algorithmic}
\end{algorithm}
\begin{proof}[Verification of certificate \CertSix]
This is a finite verification.  The dictionary records its source program,
command line, and recorded output.
The verification enumerates all labelled simple graphs on the vertex set
$\{1,\ldots,8\}$
with $m\in\{17,18\}$ edges.  For each graph it first searches for a
Hamiltonian cycle by depth-first search.  On eight vertices, a Hamiltonian
cycle is exactly a copy of $C_8$.  Hence the graphs that remain after this
filter are precisely the $C_8$-free candidates.  If no Hamiltonian cycle is found,
the program first applies the following necessary condition for planarity:
every induced subgraph on $s\ge3$ vertices has at most $3s-6$ edges.  A
graph failing this condition is safely rejected as non-planar.  The remaining
graphs are tested by an exhaustive $K_5/K_{3,3}$-minor search: recursively
deleting and contracting edges, with memoization, until either a $K_5$ or
$K_{3,3}$ subgraph is found or all minors have been exhausted.  By Wagner's
theorem~\cite{Wagner1937}, this minor test is equivalent to planarity for the
remaining graphs.

The finite verification gives the following counts:
\[
\begin{array}{c|r|r|r|r|r}
m & \#\text{graphs} & \#\text{non-Ham.}
  & \#\text{sparsity-rej.} & \#\text{minor tests}
  & \#\text{planar non-Ham.}\\
\hline
17 & 21474180 & 1311324 & 1244796 & 66528 & 0\\
18 & 13123110 & 362404 & 357308 & 5096 & 0 .
\end{array}
\]
Thus no simple planar graph on eight vertices with $17$ or $18$ edges is
non-Hamiltonian, verifying certificate \CertSix.
\end{proof}

\end{document}